\documentclass[12pt]{article}
\usepackage{geometry}
\usepackage{amssymb}
 \usepackage{xypic}
\usepackage[all,2cell]{xy}
\usepackage{graphics}
\usepackage{amsmath}
\usepackage{fullpage}
\usepackage{theorem}
\usepackage{amstext}

\newtheorem{theorem}{Theorem}[section]
\newtheorem{corollary}{Corollary}[section]
\newtheorem{proposition}{Proposition}[section]

\newtheorem{lemma}{Lemma}[section]
\theorembodyfont{\upshape}

\newtheorem{definition}{Definition}[section]

\newenvironment{proof}[1][Proof]{\textbf{#1.} }{\ \rule{0.5em}{0.5em}}
\begin{document}

\title{The Eilenberg-Mac Lane cohomology \\
of an inverse monoid and \\
the maximum group image}
\author{
Anjeza Krakulli \\
Universiteti Aleksand\"er Moisiu \\
Fakulteti i Teknologjis\"e dhe Informacionit \\
Departamenti i Matematik\"es, Durr\"es \\
anjeza.krakulli@gmail.com \\
Elton Pasku \\
Universiteti i Tiran\"es \\
Fakulteti i Shkencave Natyrore \\
Departamenti i Matematik\"es, Tiran\"e \\
elton.pasku@fshn.edu.al \\
}

\date{}

\maketitle

\begin{abstract}

The aim of this paper is to see at what extent homological properties of an inverse monoid are determined from those of its maximum group image. We provide several evidences that the maximum group image contains vital homological information which can be used to study certain properties of the monoid itself. For instance, we prove that an inverse monoid $S$ is of type $FP_{\infty}$, if and only if it contains a minimal idempotent and its maximum group image is of the same type. Regarding cohomological dimensions, we show that the cohomological dimension of a free Clifford monoid and that of its maximum group image agree and are equal to one. Also we define the index of a full submonoid of an inverse monoid in terms of their maximum group images and show that if the index is finite then, the monoid is of type $\text{FP}_{\infty}$ if and only if its submonoid is of the same type.\newline
\end{abstract}

\section{Introduction and preliminaries}

A useful way to look for homological information for an inverse monoid is to study homological properties of its maximal subgroups and see at what extent they determine certain properties of the monoid. There are several evidences given in \cite{G + P} that this approach is indeed useful. In this paper the authors have shown that a Clifford monoid $S$ is of type $FP_{n}$ if and only if $S$ contains a minimal idempotent $e$ and the maximal subgroup of $S$ containing $e$ is of type $FP_{n}$. For the wider class of inverse semigroups, they prove under the assumption that the semigroup contains a minimal idempotent, that it is of type $FP_{n}$ if and only if its maximal subgroup containing that idempotent is of the same type. Differently from \cite{G + P}, in our paper we relate homological properties of an inverse monoid $S$ to those of its maximum group image $G$. We prove that $S$ is of type $FP_{\infty}$ if and only if $S$ contains a minimal idempotent and $G$ is of type $FP_{\infty}$. We also prove that the cohomological dimension of a free Clifford monoid and that of its maximum group image agree. In this case we prove that the cohomological dimension is one which makes a free Clifford monoid another candidate to prove an analogue of the Stalling Swan theorem for inverse semigroups. At the end of the paper we define the index of a full submonoid of an inverse monoid in terms of their maximum group images and show that if the index is finite then, the monoid is of type $\text{FP}_{\infty}$ if and only if its submonoid is of the same type. This property has its counterpart in the the theory of cohomology of groups. These results provide enough evidence that the maximum group image of an inverse monoid contains vital homological information which can be used to study homological properties of the monoid itself, and therefore deserves to be studied further on. \newline
By definition $S$ is an inverse semigroup if for each element $x$ there is a unique $x^{-1}$ such that $x=xx^{-1}x$ and $x^{-1}=x^{-1}xx^{-1}$. A key property of inverse semigroups is that their idempotents commute. If $S$ is an inverse monoid and $E$ its semilattice of idempotents, then we let $G$ be the maximum group image of $S$; that is $G \cong S/\sigma$ where $\sigma$ is the congruence on $S$ defined as follows. For every $a,b \in S$, $a \sigma b$ if and only if there is an $e \in E$ such that $ae=be$, or equivalently, if there is $f \in E$ such that $fa=fb$. The unit of $G$ will be denoted by 1.
\newline
We can regard the monoids $S$ and $G$ as small categories with a single object, denoted by $\ast_{S}$ and $\ast_{G}$ respectively, and with morphisms, the elements of the respective monoids. Define $J: S \rightarrow G$ by $J(\ast_{S})=\ast_{G}$ and for every $s \in S$ we let $J(s)=\mu(s)$ where $\mu: S \rightarrow S/\sigma$ is the canonical epimorphism. One can easily prove that $J$ is a functor using the fact that it arises for an epimorphism of monoids. The functor $J$ induces a functor $J^{\ast}:\mathbf{Ab}^{G} \rightarrow \mathbf{Ab}^{S}$ by the rule $J^{\ast}(M)=MJ$, for every $G$-module $M \in \mathbf{Ab}^{G}$. \newline
The following construction is a special case of the comma category \cite{MacLane}. Denote by $\Im \downarrow \ast_{G}$ the category of $J$-objects over $\ast_{G}$ as follows. An object of $\Im \downarrow \ast_{G}$ is a pair $(\ast_{S},a)$ where $a: \ast_{G} \rightarrow \ast_{G}$ is a morphism in $G$. A morphism $s: (\ast_{S},a) \rightarrow (\ast_{S},b)$ is a morphism $s : \ast_{S} \rightarrow \ast_{S}$ such that the diagram
\begin{equation*}
\xymatrix{ J(\ast_{S})=\ast_{G} \ar[rr]^{J(s)} \ar[rd]_{a} && J(\ast_{S})=\ast_{G} \ar[ld]^{b} \\ & \ast_{G}}
\end{equation*}
commutes. In other words, there is a morphism $s: (\ast_{S},a) \rightarrow (\ast_{S},b)$ if $bJ(s)=a$. \newline
We record the following lemma for future use.
\begin{lemma} \label{e=ze}
For every monoid $S$, there is an isomorphism between additive categories $\mathbf{Ab}^{S}$ and $\mathbf{Ab}^{\mathbb{Z}S}$, where $\mathbb{Z}S$ is the additivization of $S$.
\end{lemma}
\begin{proof}
Define $\mathcal{I}_{S}: \mathbf{Ab}^{S} \rightarrow \mathbf{Ab}^{\mathbb{Z}S}$ by $\mathcal{I}_{S}(F)=\mathbb{Z}F$ on objects and by $\mathcal{I}_{S}(\tau)=\mathbb{Z}\tau$ on morphisms $\tau:F_{1} \overset{\cdot} {\rightarrow} F_{2}$. 
\end{proof}

\section{Cohomology of inverse monoids}

As long as we are trying to relate the homological properties of an inverse monoid to those of its maximum group image, it is natural to consider the Eilenberg-MacLane cohomology of monoids which by definition is given by
\begin{equation*}
H^{n}(S,M)=Ext_{\mathbb{Z}S}^{n}(\mathbb{Z},M).
\end{equation*}
Here $S$ is a monoid, $M$ is a left $S$-module and $\mathbb{Z}$ is the trivial $S$-module. \newline
The following lemma is crucial in the proof of theorem \ref{coh}.
\begin{lemma} \label{strongly filtered}
$\Im \downarrow \ast_{G}$ is filtered and if $S$ contains a minimal idempotent, then $\Im \downarrow \ast_{G}$ is strongly filtered in the sense of \cite{Sch}.
\end{lemma}
\begin{proof}
Let $(\ast_{S},a)$ and $(\ast_{S},b)$ be two objects of $\Im \downarrow {\ast_{G}}$. We can chose $\alpha$ and $\beta \in S$ such that $\mu(\alpha)=a$ and $\mu(\beta)=b$, then, from the definition $\alpha: (\ast_{S},a) \rightarrow (\ast_{S},1)$ and $\beta: (\ast_{S},b) \rightarrow (\ast_{S},1)$ are arrows in $\Im \downarrow \ast_{G}$. Secondly, if $s_{1}, s_{2}: (\ast_{S},a) \rightarrow (\ast_{S},b)$ are parallel arrows, then we have $b\mu(s_{1})=a=b\mu(s_{2})$, which means that $s_{1} \sigma s_{2}$ and as a result there is an $e \in E$ such that $es_{1}=es_{2}$. But evidently, $e: (\ast_{S},b) \rightarrow (\ast_{S},b)$ is an arrow in $\Im \downarrow \ast_{G}$, hence the above equality is an equality of arrows in $\Im \downarrow \ast_{G}$. This shows that $\Im \downarrow \ast_{G}$ is a filtered category. Lastly, we assume that $S$ contains a minimal idempotent $\varepsilon$ and let $s_{i}:(\ast_{S},a) \rightarrow (\ast_{S},a_{i})$ with $i \in I$ be a pencil in $\Im \downarrow {\ast_{G}}$. For every $i \in I$, chose $t_{i} \in S$ such that $\mu(t_{i})=a_{i}$. For every $i$ and $j \in I$ we have that $a_{i}\mu(s_{i})=a=a_{j}\mu(s_{j})$, hence $\mu(t_{i}s_{i})=\mu(t_{j}s_{j})$. Then there exists and idempotent $e_{ij} \in E$ such that $e_{ij}t_{i}s_{i}=e_{ij}t_{j}s_{j}$. Multiplying through on the left by $\varepsilon$ and recalling that $\varepsilon$ is a minimal idempotent, we obtain $\varepsilon t_{i}s_{i}=\varepsilon t_{j}s_{j}$. As before, $\varepsilon: (\ast_{S},1) \rightarrow (\ast_{S},1)$ is an arrow in $\Im \downarrow {\ast_{G}}$, therefore the family of arrows $\varepsilon t_{i}$ with $i \in I$ is a commutative completion of the given pencil showing that $\Im \downarrow {\ast_{G}}$ is strongly filtered.
\end{proof}
\newline
\newline
The following is an analogue of theorem 4.1 of \cite{Pasku} for inverse monoids in general and also an analogue of proposition 3.6 of \cite{Loganathan} which relates the Lausch cohomology of an inverse monoid to the Eilenberg-Mac Lane cohomology of its maximum group image.
\begin{theorem} \label{coh}
For every $n \geq 0$ and every left $\mathbb{Z}G$ module $M$ there is a natural isomorphism  
\begin{equation} \label{Z}
Ext_{\mathbb{Z}G}^{n}(\mathbb{Z},M) \cong Ext_{\mathbb{Z}S}^{n}(\mathbb{Z}, \mathbf{J}^{\ast}M)
\end{equation}
where $\mathbf{J}^{\ast} = \mathcal{I}_{S} J^{\ast} \mathcal{I}_{G}^{-1}$, with $\mathcal{I}_{G}$ and $\mathcal{I}_{S}$ being the isomorphisms of lemma \ref{e=ze} and $\mathbb{Z}$ is the trivial module.
\end{theorem}
\begin{proof}
We show first that $J^{\ast}$ has a left adjoint $\widetilde{J}$ by using standard categorical arguments and then we show that both, $J^{\ast}$ and $\widetilde{J}$ are exact functors. Since all colimits exist in $\mathbf{Ab}$, then the dual of theorem 1, p. 237 of \cite{MacLane} shows that every functor $T \in \mathbf{Ab}^{S}$ has a left Kan extension $Lan_{J}T$ along $J$ defined by
\begin{equation} \label{colim}
Lan_{J}T(\ast_{G})=\underrightarrow{Lim}(\Im \downarrow \ast_{G}\overset{P}{\longrightarrow} S \overset{T}{\longrightarrow} \mathbf{Ab}),
\end{equation}
where $P$ is the projection $(\ast_{S}, a) \mapsto \ast_{S}$. Now the dual of the argument given in p. 237 of \cite{MacLane} shows that the function $T \mapsto Lan_{J}T$ determines a left adjoint $\widetilde{J}$ of $J^{\ast}$. As a left adjoint, $\widetilde{J}$ preserves cokernels, so to prove it is exact it remains to show that $\widetilde{J}$ preserves monics too. Let $T_{1} \hookrightarrow T_{2}$ be a monic in $\mathbf{Ab}^{S}$, then Proposition 3.1, p. 258 of \cite{MacLane-Hom} shows that the induced morphism $T_{1}P \rightarrow T_{2}P$ is also monic in $\mathbf{Ab}^{\Im \downarrow \ast_{G}}$. Regarding $\mathbf{Ab}$ as the category of right $\mathbb{Z}$ modules and recalling from lemma \ref{strongly filtered} that $\Im \downarrow \ast_{G}$ is filtered, we can apply theorem 2.6.15 of \cite{Weib} to show that the other induced morphism $\underrightarrow{Lim}(T_{1}P) \rightarrow \underrightarrow{Lim}(T_{2}P)$ is monic which from (\ref{colim}) is the same as to say that $Lan_{J}T_{1}(\ast_{G}) \rightarrow Lan_{J}T_{2}(\ast_{G})$ is monic too. Proposition 3.1, p. 258 of \cite{MacLane-Hom} again shows that $Lan_{J}T_{1} \rightarrow Lan_{J}T_{2}$ is monic proving the exactness of $\widetilde{J}$. 

We show that also $J^{\ast}$ is exact. For this recall first that $J^{\ast}$ preserves kernels as a right adjoint. It remains to show that it preserves epics too. Indeed, if $M_{1} \twoheadrightarrow M_{2}$ is an epic in $\mathbf{Ab}^{G}$, then the induced homomorphism $M_{1}J(\ast_{S})=M_{1} \rightarrow M_{2}=M_{2}J(\ast_{S})$ is a surjective homomorphism of left $S$ modules. 

The composite $\mathbf{\widetilde{\mathbf{J}}}=\mathcal{I}_{G} \widetilde{J} \mathcal{I}_{S}^{-1} :\mathbf{Ab}^{\mathbb{Z}S} \rightarrow \mathbf{Ab}^{\mathbb{Z}G}$ is exact since $\mathcal{I}_{G}$ and $\mathcal{I}_{S}^{-1}$ are exact. In the same way we can get another exact functor $\mathbf{J}^{\ast} = \mathcal{I}_{S} J^{\ast} \mathcal{I}_{G}^{-1}:\mathbf{Ab}^{\mathbb{Z}G} \rightarrow \mathbf{Ab}^{\mathbb{Z}S}$. In fact this functor is the change of ring functor in the sense of \cite{HilStamm}. Since each isomorphism is adjoint to its inverse, we see from theorem 1, p. 103 of \cite{MacLane} that $\mathbf{\widetilde{\mathbf{J}}}$ is a left adjoint to $\mathbf{J}^{\ast}$. Now we can apply theorem 12.1, p. 162 of \cite{HilStamm} to obtain for every $n \geq 0$ a natural isomorphism
\begin{equation} \label{ext}
\Phi^{n}: Ext_{\mathbb{Z}G}^{n}(\widetilde{\mathbf{J}}N,M) \rightarrow Ext_{\mathbb{Z}S}^{n}(N, \mathbf{J}^{\ast}M)
\end{equation}
for every $M \in \mathbf{Ab}^{\mathbb{Z}G}$ and $N \in \mathbf{Ab}^{\mathbb{Z}S}$. If we take $N$ to be the trivial left $\mathbb{Z}S$ module $\mathbb{Z}$, then $\widetilde{\mathbf{J}}N$ coincides with the trivial $\mathbb{Z}G$ module $\mathbb{Z}$. Indeed, as proved in \cite{HilStamm}, the change of ring functor $\mathbf{J}^{\ast}$ has a left adjoint (which has to be naturally isomorphic to $\mathbf{\widetilde{\mathbf{J}}}$) given by the rule $N \mapsto \mathbb{Z}G \otimes_{\mathbb{Z}S} N$. For $N=\mathbb{Z}$ we see that any generator $g \otimes_{\mathbb{Z}S}z$ of $\mathbb{Z}G \otimes_{\mathbb{Z}S} \mathbb{Z}$ can be reduced as follows: $g \otimes_{\mathbb{Z}S}z=1 \otimes_{\mathbb{Z}S} s \cdot z=1 \otimes_{\mathbb{Z}S} z$ where $s \in \mu^{-1}(g)$. Therefore, $\widetilde{\mathbf{J}}\mathbb{Z} \cong \mathbb{Z}$. Applying (\ref{ext}) for $N=\mathbb{Z}$, we get the natural isomorphism $Ext_{\mathbb{Z}G}^{n}(\mathbb{Z},M) \cong Ext_{\mathbb{Z}S}^{n}(\mathbb{Z}, \mathbf{J}^{\ast}M)$.
\end{proof}
\bigskip
\newline
The isomorphism of theorem \ref{coh} shows that $\text{cd}G \leq \text{cd}S$. We prove in the following proposition that in the case of inverse monoids containing a minimal idempotent cohomological dimensions are the same and then obtain as a corollary that free clifford monoids have cohomological dimension one.
\begin{proposition} \label{min idempotent}
Let $S$ be an inverse monoid containing a minimal idempotent $\varepsilon$. Then the cohomological dimension of $S$ and that of its maximum group image $G$ agree.
\end{proposition}
\begin{proof}
We give first the strategy of the proof and then proceed with the technical details. The crux of the proof is to show that $\mathbf{J}^{\ast} \mathbb{Z}G$ is a projective left $\mathbb{Z}S$ module via $\mu$ and then utilize proposition 12.3 of \cite{HilStamm} which gives in this case a natural isomorphism $\text{Ext}^{n}_{\mathbb{Z}G}(\mathbb{Z},\text{Hom}_{\mathbb{Z}S}(\mathbf{J}^{\ast}\mathbb{Z}G,B)) \cong \text{Ext}^{n}_{\mathbb{Z}S}(\mathbb{Z},B)$ for every $n \geq 0$ and $B \in \mathbf{Ab}^{\mathbb{Z}S}$. This implies immediately that $\text{cd}S \leq \text{cd}G$. This, together with the remark after theorem \ref{coh}, imply that $\text{cd}S = \text{cd}G$. Let us show that $\mathbf{J}^{\ast} \mathbb{Z}G$ is projective. From theorem \ref{coh} we see that $\mathbf{J}^{\ast} \mathbb{Z}G$ is a flat $\mathbb{Z}S$ module, so to prove it is projective it is enough to show that it is finitely presented. Indeed, there is an exact sequence
\begin{equation} \label{fp}
 \mathbb{Z}S \overset{\psi} \rightarrow \mathbb{Z}S \overset{\tilde{\mu}} {\rightarrow} \mathbf{J}^{\ast}\mathbb{Z}G \rightarrow 0
\end{equation}
where $\tilde{\mu}$ is the linear extension of $\mu$ and $\psi$ is defined by $\psi(s)=s-s\varepsilon$. It is easy to see that $\psi$ is a left $\mathbb{Z}S$ module homomorphism which maps $\mathbb{Z}S$ onto $\text{Ker}(\tilde{\mu})$. The latter follows easily from the fact that $\text{Ker}(\tilde{\mu})$ is generated as an abelian group from elements of the form $s-s\varepsilon$. 
\end{proof} \newline
We recall from \cite{Howie} that a free Clifford monoid on a set $X$ is the set 
\begin{equation*}
CM_{X}=\{(u,A) \in FG_{X} \times \mathcal{P}(X)|c(u) \subseteq A\},
\end{equation*}
where $FG_{X}$ is the free group on $X$, $\mathcal{P}(X)$ is the set of all subsets of $X$, and by $c(u)$ we denote the content of the word $u$, that is, the set of all letters from $X$ represented in the word $u$. The multiplication on $CM_{X}$ is defined by
\begin{equation*}
(v,B)(u,A)=(vu,B \cup A).
\end{equation*}
\begin{corollary} \label{free Clifford}
If $CM_{X}$ is the free Clifford monoid on a set $X$, then the cohomological dimension of $CM_{X}$ and that of its maximum group image $G=CM_{X}/\sigma$ agree and are equal to one.
\end{corollary}
\begin{proof}
As one can see from above, a free clifford monoid contains the minimal idempotent $(1,X)$, therefore from proposition \ref{min idempotent} we have that $\text{cd}\mathbb{Z}CM_{X}=\text{cd}G$, so it remains to show that $\text{cd}G=1$. This can be achieved if we prove that $G$ is free on some set. In fact it is free on the set of equivalence classes $\{(\overline{x,\{x\}})| x \in X\}$ modulo $\sigma$. This can be proved easily using the universal property of $CM_{X}$ as depicted in the following diagram
\begin{equation*}
\xymatrix{ X  \ar@{}[d] \ar[r]^{\iota}  \ar[dr]_{f} & CM_{X} \ar@{-->}[d]^{\varphi} \ar[r]^{\mu} & CM_{X}/\sigma \ar@{-->}[ld]^{\tilde{\varphi}} \ar@{}[d]\\
{} \ar@{}[r] & H \ar@{}[r] & {}
}
\end{equation*}
where $H$ is any group and $f:X \rightarrow H$ is any map, $\iota(x)=(x,\{x\})$, $\mu$ is the natural epimorphism, $\varphi$ is the monoid morphism which extends uniquely $f$ and $\tilde{\varphi}$ is given by $\tilde{\varphi}((\overline{u,c(u)}))=\varphi((u, c(u)))$.
\end{proof}
\newline
\newline
Corollary shows that free Clifford monoids are among other candidates to prove an analogue of the Stalling Swan theorem for inverse semigroups in terms of the Eilenberg-Mac Lane cohomology. Though it should be mentioned that, similarly to the Lausch cohomology, the dimension zero case for the Eilenberg MacLane cohomology is quite different from that of groups. More specifically, for E-unitary inverse monoids we have this
\begin{corollary}
Let $S$ be an E-unitary inverse monoid. Then $S$ has cohomological dimension zero if and only if $S$ is a semilattice with zero.
\end{corollary}
\begin{proof}
From theorem \ref{coh} the maximum group image $G$ has to be zero, therefore for every $s\in S$, $(s,1_{S}) \in \sigma$. Since $S$ is E-unitary, it follows that $s$ is an idempotent. The fact that $S$ has a zero follows from \cite{Laudal} and also from \cite{Guba-Pride}. The converse is obvious.
\end{proof}
\newline
\newline
We can get another application of theorem \ref{coh} in the case of abelian inverse monoids. We recall from \cite{Nico} the following problem discussed there for abelian monoids in general. If $S$ is an abelian monoid and $T$ its maximal cancellative homomorphic image, then for any $\mathbb{Z}S$ module $D$ we form the $\mathbb{Z}T$ module $D'=\text{Hom}_{\mathbb{Z}S}(\mathbb{Z}T,D)$. It is well known the existence of a natural homomorphism $H^{n}(T,D') \rightarrow H^{n}(S,D)$. If $D$ is a trivial $\mathbb{Z}S$ module, then $D'$ is just $D$ regarded as a trivial $\mathbb{Z}T$ module. The above homomorphism turns out to be an isomorphism for $D$ trivial and $n=1,2$ but it was unknown what happens for $n \geq 3$. If $S$ is an abelian inverse monoid, then it is obvious that $T$ is just $G$, the maximum group image of $S$. In this case we can utilize the isomorphism (\ref{Z}) to obtain an isomorphism $Ext_{\mathbb{Z}G}^{n}(\mathbb{Z},D') \cong Ext_{\mathbb{Z}S}^{n}(\mathbb{Z}, \mathbf{J}^{\ast}D')$. The following is now immediate.
\begin{corollary}
If $S$ is an abelian inverse monoid and $G$ its maximum group image, then for any trivial $\mathbb{Z}S$ module $D$ there is a natural isomorphism $H^{n}(S,D) \cong H^{n}(G,D)$ for every $n \in \mathbb{N}$.
\end{corollary} 
The proof of theorem \ref{S-G} will use a characterization of the $FP_{\infty}$ property for modules $M$ in terms of the functors $\text{Ext}(M,\bullet)$ and $\text{Tor}(M,\bullet)$.

\begin{theorem} \label{S-G}
Let $S$ be an inverse monoid and $G$ its maximum group image. If $S$ is of type $FP_{\infty}$, then $S$ contains a minimal idempotent and $G$ is of type $FP_{\infty}$. Conversely, if $S$ contains a minimal idempotent and $G$ is of type $FP_{\infty}$, then $S$ is of type $FP_{\infty}$.
\end{theorem}
\begin{proof}
In particular $S$ is of type $FP_{1}$, therefore there is free partial resolution of finite type of the trivial $\mathbb{Z}S$ module $\mathbb{Z}$:
\begin{equation} \label{FP1}
\oplus_{i \in I_{1}} \mathbb{Z}S \rightarrow \oplus_{i \in I_{0}} \mathbb{Z}S \rightarrow \mathbb{Z} \rightarrow 0.
\end{equation}
We can regard $\mathbb{Z}E$ as a right $\mathbb{Z}S$ module via the conjugation of $S$ on $E$: $e \cdot s=s^{-1}es$ for all $s \in S$ and $e \in E$. If we tensor (\ref{FP1}) on the left by $\mathbb{Z}E$ regarded as a left $\mathbb{Z}E$ module and a right $\mathbb{Z}S$ module, we obtain the following exact sequence in $\mathbb{Z}E\text{-}\mathbf{Mod}$
\begin{equation} \label{tensor FP1}
 \oplus_{i \in I_{1}} \mathbb{Z}E \otimes_{\mathbb{Z}S} \mathbb{Z}S \rightarrow  \oplus_{i \in I_{0}} \mathbb{Z}E \otimes_{\mathbb{Z}S} \mathbb{Z}S \rightarrow \mathbb{Z}E \otimes_{\mathbb{Z}S} \mathbb{Z} \rightarrow 0.
\end{equation}
Next we show that $\mathbb{Z}E \otimes_{\mathbb{Z}S} \mathbb{Z}$ is the trivial $\mathbb{Z}E$ module $\mathbb{Z}$. Indeed, any generator $e \otimes_{\mathbb{Z}S}z$ can be reduced as follows: $e \otimes_{\mathbb{Z}S}z=1_{S} \otimes_{\mathbb{Z}S}e\cdot z=1_{S} \otimes_{\mathbb{Z}S} z$, hence $\mathbb{Z}E \otimes_{\mathbb{Z}S} \mathbb{Z} \cong \mathbb{Z}$ as abelian groups. On the other hand, the action of any idempotent $f$ on the generator $1_{S} \otimes_{\mathbb{Z}S} z$ leaves fixed this generator. Similarly one can show that $\mathbb{Z}E \otimes_{\mathbb{Z}S} \mathbb{Z}S \cong \mathbb{Z}E$. Indeed, any generator $e \otimes_{\mathbb{Z}S}s$ of $\mathbb{Z}E \otimes_{\mathbb{Z}S} \mathbb{Z}S$ can be reduced as follows: $e \otimes_{\mathbb{Z}S}s=s^{-1}es \otimes_{\mathbb{Z}S} 1_{S}$ which shows that the elements of $\mathbb{Z}E \otimes_{\mathbb{Z}S} \mathbb{Z}S$ are $\mathbb{Z}$-linear combinations of elements of the form $e \otimes_{\mathbb{Z}S} 1_{S}$. From this it is easy to see why $\mathbb{Z}E \otimes_{\mathbb{Z}S} \mathbb{Z}S \cong \mathbb{Z}E$ as left $\mathbb{Z}E$ modules. Now using the exactness of (\ref{tensor FP1}) and above isomorphisms, we can obtain the following exact sequence
\begin{equation*} 
\oplus_{i \in I_{1}} \mathbb{Z}E \rightarrow \oplus_{i \in I_{0}} \mathbb{Z}E \rightarrow \mathbb{Z} \rightarrow 0,
\end{equation*}
proving that $E$ is of type $FP_{1}$. The main result of \cite{Kobayashi FP1} shows that $E$ is unitarily finitely generated and then in the same way as in the proof of theorem 9 of \cite{G + P} we can show that $E$ contains a minimal idempotent.
\newline
To prove that $G$ is of type $FP_{\infty}$, as shown in \cite{Brown-book} (see also \cite{Brown}), we need to prove that $Ext_{\mathbb{Z}G}^{n}(\mathbb{Z},\bullet)$ commutes with direct limits. Let $\underrightarrow{lim}M_{i}$ be a direct limit of a diagram of modules $M_{i}$ with $i \in I$. Then we have the following natural isomorphisms
\begin{align*}
Ext_{\mathbb{Z}G}^{n}(\mathbb{Z},\underrightarrow{Lim}M_{i}) & \cong Ext_{\mathbb{Z}S}^{n}(\mathbb{Z}, \mathbf{J}^{\ast}\underrightarrow{Lim}M_{i}) && \text{ from (\ref{Z})} \\
& \cong Ext_{\mathbb{Z}S}^{n}(\mathbb{Z}, \underrightarrow{Lim}  \mathbf{J}^{\ast}M_{i}) && \text{ the functor $\mathbf{J}^{\ast}$ has a right adjoint \cite{HilStamm}}\\
& \cong \underrightarrow{Lim} Ext_{\mathbb{Z}S}^{n}(\mathbb{Z}, \mathbf{J}^{\ast} M_{i}) && \text{ $S$ is of type $FP_{\infty}$ }\\
& \cong \underrightarrow{Lim}Ext_{\mathbb{Z}G}^{n}(\mathbb{Z},M_{i}) && \text{ from the naturality of (\ref{Z})},
\end{align*}
which proves that $G$ is of type $FP_{\infty}$ as required.\newline
Before we prove the converse under the given hypothesis, we will make an observation. We denote by $\mathbb{Z}$ the trivial right $\mathbb{Z}S$ module $\mathbb{Z}$ and by $\mathbb{Z}'$ the trivial right $\mathbb{Z}G$ module $\mathbb{Z}$. For every abelian group $C$, the left $\mathbb{Z}S$ modules $\mathbf{J}^{\ast} \text{Hom}_{\mathbb{Z}}(\mathbb{Z}',C)$ and $\text{Hom}_{\mathbb{Z}}(\mathbb{Z},C)$ coincide. Of course, as abelian groups they are equal. To see that they are equal as modules, we recall that the action of $S$ on the elements $f$ of $\mathbf{J}^{\ast}\text{Hom}_{\mathbb{Z}}(\mathbb{Z}',C)$ is defined by posing $s\cdot f(x)=\mu(s) \cdot f(x)=f(x\cdot \mu(s))$ for every $s \in S$ and $x \in \mathbb{Z}'$. But from the definition of $\mathbb{Z}'$, $f(x\cdot \mu(s))=f(x)$, and as a result $s\cdot f=f$ which means that $\mathbf{J}^{\ast} \text{Hom}_{\mathbb{Z}}(\mathbb{Z}',C)$ is trivial as a left $\mathbb{Z}S$ module. Similarly, we can prove that the abelian group $\text{Hom}_{\mathbb{Z}}(\mathbb{Z},C)$ is trivial as a left $\mathbb{Z}S$ module, therefore we have the equality. For any left $\mathbb{Z}S$ module $A$ and every abelian group $C$ we have the following natural isomorphisms.
\begin{align*}
\text{Hom}_{\mathbb{Z}}(\mathbb{Z}' \otimes_{\mathbb{Z}G} \widetilde{\mathbf{J}}A,C) & \cong \text{Hom}_{\mathbb{Z}G}(\widetilde{\mathbf{J}}A, \text{Hom}_{\mathbb{Z}}(\mathbb{Z}',C)) && \text{ from the adjoint associativity } \\
& \cong \text{Hom}_{\mathbb{Z}S}(A, \mathbf{J}^{\ast} \text{Hom}_{\mathbb{Z}}(\mathbb{Z}',C)) && \text{ from theorem \ref{coh}} \\
& = \text{Hom}_{\mathbb{Z}S}(A, \text{Hom}_{\mathbb{Z}}(\mathbb{Z},C)) && \text{ from our observation } \\
& \cong \text{Hom}_{\mathbb{Z}}(\mathbb{Z} \otimes_{\mathbb{Z}S}A, C) && \text{ from the adjoint associativity.}
\end{align*} 
It follows that we have the natural isomorphism in $\mathbf{Ab}$
\begin{equation*}
\text{Nat}(\text{Hom}_{\mathbb{Z}}(\mathbb{Z}' \otimes_{\mathbb{Z}G} \widetilde{\mathbf{J}}A,\bullet), \mathbb{Z}\otimes_{\mathbb{Z}} \bullet) \cong \text{Nat}(\text{Hom}_{\mathbb{Z}}(\mathbb{Z} \otimes_{\mathbb{Z}S}A, \bullet), \mathbb{Z}\otimes_{\mathbb{Z}} \bullet).
\end{equation*}
Yoneda lemma now implies the existence of a natural isomorphism
\begin{equation*}
\mathbb{Z}' \otimes_{\mathbb{Z}G} \widetilde{\mathbf{J}}A \cong \mathbb{Z} \otimes_{\mathbb{Z}S}A.
\end{equation*}
Since $\widetilde{\mathbf{J}}$ preserves projective resolutions (theorem 12.1, p. 162 of \cite{HilStamm}), the naturality of the above isomorphism implies that for every $n \geq 0$ there are natural isomorphisms
\begin{equation} \label{tor}
\text{Tor}^{\mathbb{Z}G}_{n}(\mathbb{Z}', \widetilde{\mathbf{J}}A) \cong \text{Tor}^{\mathbb{Z}S}_{n}(\mathbb{Z},A).
\end{equation}
Similarly one can prove that for any $K \in \mathbf{Ab}^G$ there is a natural isomorphism
\begin{equation} \label{torG}
\text{Tor}_{n}^{\mathbb{Z}G}(\mathbb{Z}', \mathcal{I}_{G}K) \cong \text{Tor}_{n}^{G}(\mathcal{I}_{G}^{-1}\mathbb{Z}',K).
\end{equation}
To prove that the trivial $\mathbb{Z}S$ module $\mathbb{Z}$ is of type $FP_{\infty}$ we must show that $\text{Tor}^{\mathbb{Z}S}_{n}(\mathbb{Z}, \bullet)$ commutes with direct products. Let $A_{i}$ for $i \in I$ be a family of left $\mathbb{Z}S$ modules. The following natural isomorphisms hold true.
\begin{align*}
\text{Tor}^{\mathbb{Z}S}_{n}(\mathbb{Z}, \underset{i \in I}{\Pi}A_{i}) & \cong \text{Tor}^{\mathbb{Z}G}_{n} (\mathbb{Z}', \widetilde{\mathbf{J}} \underset{i \in I}{\Pi}A_{i}) && \text{ from (\ref{tor})} \\
& \cong \text{Tor}^{\mathbb{Z}G}_{n} (\mathbb{Z}',\mathcal{I}_{G} \widetilde{J} \mathcal{I}_{S}^{-1}(\underset{i \in I}{\Pi}A_{i})) && \text{ from the definition of $\widetilde{\mathbf{J}}$} \\
& \cong \text{Tor}^{\mathbb{Z}G}_{n} (\mathbb{Z}', \mathcal{I}_{G} \widetilde{J} (\underset{i \in I}{\Pi} \mathcal{I}_{S}^{-1}A_{i})  \\
& \cong \text{Tor}^{\mathbb{Z}G}_{n} (\mathbb{Z}', \mathcal{I}_{G} \underrightarrow{Lim}((\underset{i \in I}{\Pi}\mathcal{I}_{S}^{-1}A_{i})P)) && \text{ from theorem \ref{coh}} \\
& \cong \text{Tor}^{\mathbb{Z}G}_{n} (\mathbb{Z}',\mathcal{I}_{G} \underrightarrow{Lim}(\underset{i \in I}{\Pi}(\mathcal{I}_{S}^{-1}A_{i}P))) && \text{ $(\underset{i \in I}{\Pi}\mathcal{I}_{S}^{-1}A_{i})P \cong \underset{i \in I}{\Pi}(\mathcal{I}_{S}^{-1}A_{i}P)$} \\
& \cong \text{Tor}^{G}_{n} (\mathcal{I}_{G}^{-1}\mathbb{Z}', \underrightarrow{Lim}(\underset{i \in I}{\Pi}(\mathcal{I}_{S}^{-1}A_{i}P))) && \text{ from (\ref{torG})} \\
& \cong \text{Tor}^{G}_{n} (\mathcal{I}_{G}^{-1}\mathbb{Z}', \underset{i \in I}{\Pi} \underrightarrow{Lim}\mathcal{I}_{S}^{-1}A_{i}P) && \text{ lemma \ref{strongly filtered} and proposition 9.5.3 of \cite{Sch}} \\
& \cong \text{Tor}^{\mathbb{Z}G}_{n} (\mathbb{Z}', \mathcal{I}_{G} \underset{i \in I}{\Pi} \underrightarrow{Lim}\mathcal{I}_{S}^{-1}A_{i}P) \\
& \cong \text{Tor}^{\mathbb{Z}G}_{n} (\mathbb{Z}',  \underset{i \in I}{\Pi} \mathcal{I}_{G} \underrightarrow{Lim}\mathcal{I}_{S}^{-1}A_{i}P) && \text{ $\mathcal{I}_{G}$ is a right adjoint} \\
& \cong \underset{i \in I}{\Pi} \text{Tor}^{\mathbb{Z}G}_{n} (\mathbb{Z}', \mathcal{I}_{G}\underrightarrow{Lim}\mathcal{I}_{S}^{-1}A_{i}P) && \text{ $G$ is of type $FP_{\infty}$} \\
& \cong \underset{i \in I}{\Pi} \text{Tor}^{G}_{n} (\mathcal{I}_{G}^{-1}\mathbb{Z}', \underrightarrow{Lim}\mathcal{I}_{S}^{-1}A_{i}P) \\
& \cong \underset{i \in I}{\Pi} \text{Tor}^{G}_{n} (\mathcal{I}_{G}^{-1}\mathbb{Z}', \widetilde{J}\mathcal{I}_{S}^{-1}A_{i}) \\
& \cong \underset{i \in I}{\Pi} \text{Tor}^{\mathbb{Z}G}_{n} (\mathbb{Z}', \mathcal{I}_{G} \widetilde{J}\mathcal{I}_{S}^{-1}A_{i}) \\
& \cong \underset{i \in I}{\Pi} \text{Tor}^{\mathbb{Z}G}_{n} (\mathbb{Z}', \widetilde{\mathbf{J}} A_{i}) && \text{ from the definition of $\widetilde{\mathbf{J}}$} \\
& \cong \underset{i \in I}{\Pi} \text{Tor}^{\mathbb{Z}S}_{n}(\mathbb{Z}, A_{i}) && \text{ from (\ref{tor}).}
\end{align*}
Now comparing the left with the right hand side in the above sequence of isomorphisms, we get the result.
\end{proof}
\newline
\newline
We use the above result to show that property $FP_{\infty}$ behaves nicely with respect to inverse subsemigroups of finite index in the following sense.
\begin{definition}
Let $H$ be a full inverse subsemigroup of an inverse monoid $S$. We say that $H$ is of finite index in $S$ if the maximum group image of $H$ has finite index in the maximum group image of $S$.
\end{definition}
This definition makes sense since the maximum group image of $H$ can be regarded as a subgroup of the maximum group image of $S$ as one can easily check.
\begin{proposition}
Let $S$ be an inverse monoid and $H$ be an inverse subsemigroup of $S$ of finite index. Then, $S$ is of type $FP_{\infty}$ if and only if $H$ is of the same type.
\end{proposition}
\begin{proof}
If $S$ if of type $FP_{\infty}$ then it contains a minimal idempotent $e$ and its maximum group image $\mathcal{S}$ is of the same type. Since $H$ is a full inverse subsemigroup of $S$, it will contain $e$. On the other hand, as we mentioned before, $\mathcal{H} \leq \mathcal{S}$ where $\mathcal{H}$ is the maximum group image of $H$, and from the condition $[\mathcal{S}:\mathcal{H}] < \infty$. Proposition 5.1 of \cite{Brown-book} now implies that $\mathcal{H}$ is of type $FP_{\infty}$ and therefore $H$ is of that type. Conversely, if $H$ is of type $FP_{\infty}$, then $H$ and therefore $S$ contains a minimal idempotent. On the other hand, $\mathcal{H}$ is of type $FP_{\infty}$, hence $\mathcal{S}$ is of the same type as $[\mathcal{S}:\mathcal{H}] < \infty$. The result of theorem \ref{S-G} finishes the proof.
\end{proof}

\end{document}